\documentclass{amsart}

\newcommand{\C}{\mathbb{C}}
\newcommand{\D}{\mathbb{D}}
\newcommand{\Z}{\mathbb{Z}}
\newcommand{\R}{\mathbb{R}}
\newcommand{\N}{\mathbb{N}}

\newcommand{\dom}{\operatorname{dom}}

\newcommand{\Log}{\operatorname{Log}}
\newcommand{\Arg}{\operatorname{Arg}}
\renewcommand{\Re}{\operatorname{Re}}
\renewcommand{\Im}{\operatorname{Im}}

\newcommand{\norm}[1]{\| #1 \|}

\newtheorem{theorem}{Theorem}[section]
\newtheorem{lemma}[theorem]{Lemma}

\theoremstyle{definition}
\newtheorem{definition}[theorem]{Definition}

\theoremstyle{theorem}
\newtheorem{corollary}[theorem]{Corollary}

\theoremstyle{theorem}
\newtheorem{proposition}[theorem]{Proposition}

\theoremstyle{theorem}

\theoremstyle{theorem}

\theoremstyle{definition}

\theoremstyle{theorem}

\numberwithin{equation}{section}

\newtheorem*{carleson}{Carleson's Theorem}

\begin{document}
\title{Algorithmic randomness and Fourier analysis}
\author[Franklin]{Johanna N.Y.\ Franklin}
\address[Franklin]{Department of Mathematics \\ Room 306, Roosevelt Hall \\ Hofstra University \\ Hempstead, NY 11549-0114 \\ USA}
\email{johanna.n.franklin@hofstra.edu}
\urladdr{http://people.hofstra.edu/Johanna\_N\_Franklin/}

\author[McNicholl]{Timothy H.\ McNicholl}
\address[McNicholl]{Department of Mathematics\\ Iowa State University\\ Ames, Iowa 50011}
\email{mcnichol@iastate.edu}

\author[Rute]{Jason Rute}
\address[Rute]{Department of Mathematics\\
	Pennsylvania State University\\
	University Park, PA 16802}
\email{jmr71@math.psu.edu}
\urladdr{http://www.personal.psu.edu/jmr71/}

\begin{abstract}
Suppose $1 < p < \infty$.  
Carleson's Theorem states that the Fourier series of any function in $L^p[-\pi, \pi]$ converges almost everywhere.   We show that the Schnorr random points are precisely those that satisfy this theorem for every $f \in L^p[-\pi, \pi]$ given natural computability conditions on $f$ and $p$.
\end{abstract}
\maketitle

\section{Introduction}\label{sec:intro}

Recent discoveries have shown that algorithmic randomness has a very natural connection with classical analysis. Many theorems in analysis have the form ``For almost every $x$, $\ldots$"; the set of points for which the central claim of the theorem fails for a given choice of parameters is called an \emph{exceptional set} of the theorem.  For example, one of Lebesgue's differentiation theorems states that if $f$ is a monotone function on $[0,1]$, then $f$ is differentiable almost everywhere.  In this case, for each monotone function $f$ on $[0,1]$, the set of points at which $f$ is not differentiable is an exceptional set.  On the other hand, every natural randomness notion is characterized by a conull class of points. This suggests it is possible to characterize the points that satisfy a particular theorem in analysis in terms of a randomness notion.  Put another way, it may be the case that exceptional sets of a theorem can be used to characterize a standard notion of randomness.

To date, results of this nature have been discovered in ergodic theory \cite{bdhms12,fgmn,ft-mp,ghr11,hoyrup13,mnz,v97}, differentiability \cite{bhmn-14,bmn16,Freer.Kjos-Hanssen.Nies.ea:2014,Hoyrup.Rojas:2009b,mnz,nies14,Pathak.Rojas.Simpson:2014}, Brownian motion \cite{abs14,Asarin.Pokrovskii:1986,Fouche:2000}, and other topics in analysis \cite{av13,cf-ud,Rute:2013pd}. In this paper, we add Fourier series to this list by considering Carleson's Theorem. The original version of this theorem was proven in 1966 by L. Carleson for $L^2$ functions \cite{carleson66}; we will consider an extension of this theorem to $L^p$ functions for $p>1$ that is due to Hunt but still generally referred to as Carleson's Theorem \cite{hunt68}.  Throughout this paper we only consider the complex version of $L^p[-\pi, \pi]$; that is, we work in the space of all measurable 
$f : [-\pi, \pi] \rightarrow \C$ so that $\int_{-\pi}^\pi |f(t)|^p\ dt < \infty$.

\begin{carleson}
Suppose $1 < p < \infty$.  
 If $f$ is a function in $L^p[-\pi, \pi]$, then the Fourier series of $f$ converges to $f$ almost everywhere.
\end{carleson}

Suppose $1 < p < \infty$.  It is well known that the Fourier series of any $f \in L^p[-\pi, \pi]$ 
converges to $f$ in the $L^p$-norm.  It follows that if the Fourier series of $f \in L^p[-\pi, \pi]$ converges almost everywhere, then it converges to $f$ almost everywhere.

We consider Carleson's Theorem in the context of computable analysis and demonstrate the points that satisfy this theorem are precisely the Schnorr random points via the following two theorems.

\begin{theorem}\label{thm:main.1}
Suppose $p > 1$ is a computable real.  
If $t_0 \in [-\pi, \pi]$ is Schnorr random and $f$ is a computable vector in $L^p[-\pi, \pi]$, then the Fourier series for $f$ converges at $t_0$. 
\end{theorem}

\begin{theorem}\label{thm:main.2}
If $t_0 \in [-\pi, \pi]$ is not Schnorr random, then there is a computable function $f : [-\pi, \pi] \rightarrow \C$ whose Fourier series diverges at $t_0$.	
\end{theorem}

It is well known that when $p \geq 1$ is a computable real, there are incomputable functions in $L^p[-\pi, \pi]$ that are nevertheless computable as vectors, e.g., step functions.  Thus, Theorem \ref{thm:main.2} is considerably stronger than the converse of Theorem \ref{thm:main.1}. 
To the best of our knowledge, Theorems \ref{thm:main.1} and \ref{thm:main.2} yield the first characterization of a randomness notion via a theorem of Fourier analysis.  The proofs reveal some interesting and sometimes surprising connections between topics from algorithmic randomness such as Schnorr integral tests and topics from classical analysis such as analytic and harmonic function theory.

The paper is organized as follows.  In Section \ref{sec:bg}, we present the necessary background.     Sections \ref{sec:1stproof} and \ref{sec:2ndproof} contain the proofs of Theorems \ref{thm:main.1} and \ref{thm:main.2}, respectively.  In Sections \ref{sec:conv.to.ft0} and \ref{sec:Fejer.Lebesgue} we give two variations of Theorem \ref{thm:main.1}.  The first variation characterizes the values to which the Fourier series converges.  The second variation addresses the Fej\'er-Lebesgue Theorem which is similar to Carleson's Theorem, but also applies to the $L^1$ case.   Section \ref{sec:conclusion} contains a broader analysis of our results.

\section{Background and preliminaries}\label{sec:bg}

We begin with the necessary topics from analysis and then discuss computable analysis and algorithmic randomness.  We assume the reader is familiar with classical computability in discrete settings as expounded in \cite{Cooper.2004,o1,o2,soare}.  

\subsection{Fourier analysis}

We begin with some notation. For all $n \in \Z$ and $t \in [-\pi, \pi]$, let $e_n(t) =
e^{in t}$.
For all $n \in \Z$ and $f \in L^1[-\pi, \pi]$, let 
\[
c_n(f) = \frac{1}{2\pi} \int_{-\pi}^\pi f(t) e^{i n t} dt.
\] 
For all $f \in L^1[-\pi, \pi]$ and all $N \in \N$, let
\[
S_N(f) = \sum_{n = -N}^N c_n(f) e_n.
\]
That is, $S_N(f)$ is the $(N+1)^{st}$ partial sum of the Fourier series of $f$.  
We say that $f \in L^1[-\pi, \pi]$ is \emph{analytic} if $c_n(f) = 0$ whenever $n < 0$.

C. Fefferman showed that when $1 < p < \infty$, there is a constant $C$ so that 
\[
\norm{\sup_N |S_N(f)|}_1 \leq C \norm{f}_p
\]
for all $f \in L^p[-\pi, \pi]$ \cite{Fefferman.1973,Fefferman.1997}.  We can (and do) assume that $C$ is a positive integer.  The operator $f \mapsto \sup_N |S_N(f)|$ is known as the \emph{Carleson operator}.

Let $E = \{e_n\ :\ n \in \Z\}$.  A \emph{trigonometric polynomial}  is a function in the linear span of $E$.  
If $p$ is a trigonometric polynomial, then the \emph{degree} of $p$
is the smallest $d \in \N$ so that $S_d(p) = p$.

\subsection{Complex analysis}

We now summarize the required information on analytic and harmonic functions, in particular harmonic measure.  This material will be used exclusively in Section \ref{sec:2ndproof} (the proof of Theorem \ref{thm:main.2}).  
More expansive treatments of analytic and harmonic functions can be found in \cite{Conway.1978} and \cite{Nehari.1952}; the material on harmonic functions is drawn from \cite{Garnett.Marshall.2005}.

Suppose $U \subseteq \C$ is open and connected.  Recall that a function $f : U \rightarrow \C$ is \emph{analytic} if it has a power series expansion at each point of $U$; equivalently, if $f$ is differentiable at each $z_0 \in U$ in the sense that 
\[
\lim_{z \rightarrow z_0} \frac{f(z) - f(z_0)}{z - z_0}
\]
exists.

Let $\D$ denote the unit disk, and let $\lambda$ denote Lebesgue measure on the unit circle.  The points on the unit circle are called the \emph{unimodular points}.  
When $f$ is analytic on $\D$, let 
\[
a_n(f) = \frac{f^{(n)}(0)}{n!}
\]
for all $n \in \N$.  That is, $a_n(f)$ is the $(n+1)^{st}$ coefficient of the MacLaurin series of $f$.
Thus, 
\[
f(z) = \sum_{n = 0}^\infty a_n(f) z^n
\]
for all $z \in \D$.

Now we turn our attention to harmonic functions. Again, let $U$ be a subset of the plane that is open and connected.  Recall that a function 
$u : U \rightarrow \R$ is \emph{harmonic} if it is twice continuously differentiable and satisfies Laplace's equation 
\[
\frac{\partial^2 u}{\partial x^2} + \frac{\partial^2 u}{\partial y^2} = 0.
\]
When $u$ is harmonic on $\D$, let $\tilde{u}$ denote the harmonic conjugate of 
$u$ that maps $0$ to $0$.  
That is, $\tilde{u}$ is the harmonic function on $\D$ so that $\tilde{u}(0) = 0$ and so that 
$u$ and $\tilde{u}$ satisfy the Cauchy-Riemann equations:
\[
\frac{\partial u}{\partial x} = \frac{\partial \tilde{u}}{\partial y}\ \ \ \frac{\partial u}{\partial y} = - \frac{\partial \tilde{u}}{\partial x}.
\]
Let $\hat{u} = u + i \tilde{u}$.  Thus, $\hat{u}$ is analytic and is called the \emph{analytic extension} of $u$.

When $B$ is a Borel subset of the unit circle, there is a harmonic function $u$ on the unit disk so that 
for all unimodular $\zeta$, $\lim_{z \rightarrow \zeta} u(z) = \chi_B(\zeta)$ (where $\chi_A$ denotes the characteristic function of $A$); let $\omega(z, B, \D) = u(z)$.  The quantity $\omega(z, B, \D)$ is called the 
\emph{harmonic measure of $B$ at $z$}.  For each $z$ in the unit disk, $\omega(z, \cdot, \D)$ is
a Borel probability measure on the unit circle.  Moreover, $\omega(0, B, \D) = (2\pi)^{-1}\lambda(B)$ \cite{Garnett.Marshall.2005}.

An explicit formula for the harmonic measure of an open arc on the unit circle can be obtained 
as follows.  Let $\Log$ denote the principal branch of the complex logarithm.  That is, 
\[
\Log(z) = \int_1^z \frac{1}{\zeta} d\zeta
\]
for all points $z$ that do not lie on the negative real axis.  Let $\Arg = \Im(\Log)$.  
Suppose $A = \{e^{i\theta}\ :\ \theta_1 < \theta < \theta_2\}$ where $-\pi < \theta_1 < \theta_2 <\pi$.  
Then
\[
\omega(z, A, \D) = \frac{1}{\pi} \Arg\left( \frac{z - e^{i \theta_2}}{z - e^{i \theta_1}}\right) - \frac{1}{2\pi}(\theta_2 - \theta_1).
\]
(See Exercise 1 on p.\ 26 of \cite{Garnett.Marshall.2005}.)  It follows that 
\begin{eqnarray}
\tilde{\omega}(z, A, \D) & = & \frac{1}{\pi} \ln\left| \frac{z - e^{i \theta_2}}{z - e^{i \theta_1}} \right|\mbox{ and}\\
\hat{\omega}(z, A, \D) & = & \frac{1}{\pi i} \Log\left( \frac{z - e^{i \theta_2}}{z - e^{i \theta_1}} \right) - \frac{1}{2\pi} (\theta_2 - \theta_1).\label{eqn:omega.hat}
\end{eqnarray}

\subsection{Computable analysis}

We now use the classical concepts of computability in a discrete setting to define the concept of computability in a continuous setting. 

A complex number $z$ is \emph{computable} if there is an algorithm that, given a nonnegative integer 
$k$ as input, computes a rational point $q$ so that $|q - z| < 2^{-k}$.  A sequence $\{a_n\}_{n \in \Z}$ of points in the plane is \emph{computable} if there is an algorithm that, given an $n \in \Z$ and a $k \in \N$ as input, computes a rational point $q$ so that $|a_n - q| < 2^{-k}$.

Let us call a trigonometric polynomial $\tau$ \emph{rational} if each of its coefficients is a 
rational point.

\begin{definition}\label{def:comp.vector}
Suppose $p \geq 1$ is a computable real and suppose $f \in L^p[-\pi, \pi]$.  Then $f$ is a \emph{computable vector of $L^p[-\pi, \pi]$} if there is an algorithm that, given $k \in \N$ as input, computes a rational polynomial $\tau$ so that $\norm{f - \tau}_p < 2^{-k}$.
\end{definition}

In other words, a vector $f \in L^p[-\pi, \pi]$ is computable if it is possible to compute 
arbitrarily good approximations of $f$ by rational trigonometric polynomials.

The next proposition states the fundamental computability results we shall need about vectors 
in $L^p[-\pi, \pi]$.  

\begin{proposition}\label{prop:Fourier.comp}
Suppose $p \geq 1$ is a computable real and $f \in L^p[-\pi, \pi]$.  
\begin{enumerate}
	\item If $f$ is a computable vector, then $\norm{f}_p$ and $\{c_n(f)\}_{n \in \Z}$ are computable. \label{prop:Fourier.comp::itm:transform.norm}
	
	\item If $p = 2$, then $f$ is computable if both $\norm{f}_2$ and $\{c_n(f)\}_{n \in \Z}$ are 
computable.  \label{prop:Fourier.comp::itm:converse.L2}
\end{enumerate}
\end{proposition}

\begin{proof}
Suppose $\tau$ is a rational trigonometric polynomial.  The $L^p$-norm of $\tau$ can be computed directly from $\tau$.  
Since $| \norm{f}_p - \norm{\tau}_p| \leq \norm{f - \tau}_p$, it follows that $\norm{f}_p$ is computable.
We also have
\begin{align*}
|c_n(f) - c_n(\tau)|^p & = \left| \int_{-\pi}^\pi (f(\theta) - \tau(\theta))e^{i n \theta}\frac{d\theta}{2\pi} \right|^p \\
& \leq \left(\int_{-\pi}^\pi |f(\theta) - \tau(\theta)| \frac{d\theta}{2\pi}\right)^p\\
& \leq \int_{-\pi}^\pi \left|f(\theta) - \tau(\theta)\right|^p\frac{d\theta}{2\pi}=\norm{f - \tau}_p^p,
\end{align*}
where the last step is by Jensen's Inequality. 
It follows that $\{c_n(f)\}_{n \in \Z}$ is computable.

Now suppose $p = 2$ and suppose $\{c_n(f)\}_{n \in \Z}$ and $\norm{f}_2$ are computable.  Since $\norm{f}_2^2 = \sum_{n \in \Z} |c_n(f)|^2$ and $\norm{f - S_N(f)}_2^2 = \sum_{|n| > N} |c_n(f)|^2$, it follows
that $f$ is a computable vector in $L^2[-\pi, \pi]$.  
\end{proof}

The following corollary shows that the computability of a vector in $L^p[-\pi, \pi]$ is distinct 
from the computability of its Fourier coefficients.

\begin{corollary}
There is an incomputable vector $f \in L^2[-\pi, \pi]$ so that $\{c_n(f))\}_{n \in \Z}$ is computable.
\end{corollary}

\begin{proof}
Let $\{r_n\}_{n \in \N}$ be any computable sequence of positive rational numbers so that 
$\sum_{n = 0}^\infty r_n^2$ is incomputable.  (The existence of such a sequence follows from the constructions of E. Specker \cite{Specker.1949}.)  Set $f = \sum_{n = 0}^\infty r_n e_n$. Then $\norm{f}_2^2 = \sum_{n = 0}^\infty r_n^2$ is incomputable.  Thus, by Proposition \ref{prop:Fourier.comp}, $f$ is incomputable. 
\end{proof}

We now discuss computability of planar sets and functions.  A comprehensive treatment of the computability of functions and sets in continuous settings can be found in \cite{Weihrauch.2000}; the reader may also see \cite{Turing.1937}, \cite{Grzegorczyk.1957}, \cite{Lacombe.1955.a}, \cite{Lacombe.1955.b}, \cite{Brattka.Weihrauch.1999}, \cite{Pour-El.Richards:1989}, and \cite{Braverman.Cook.2006}.  To begin, an interval is \emph{rational} if its endpoints are rational numbers.   
An \emph{open (closed) rational rectangle} is a Cartesian product of open (closed) rational intervals. 

An open subset of the plane $U$ is \emph{computably open} if it is open and the set of all 
closed rational rectangles that are included in $U$ is computably enumerable.  
On the other hand, an open subset of the real line $X$ is \emph{computably open} 
if the set of all closed rational \emph{intervals} that are included in $X$ is computably open.  
A sequence of open sets of reals $\{U_n\}_{n \in \N}$ is \emph{computable} if $U_n$ is computably open uniformly in $n$; that is, if there is an algorithm that, given any $n \in \N$ as input, 
produces an algorithm that enumerates the closed rational intervals included in $U_n$.

Suppose $X$ is a compact subset of the plane.  A \emph{minimal cover} of $X$ is a 
finite sequence of open rational rectangles $(R_0, \ldots, R_m)$ so that 
$X \subseteq \bigcup_j R_j$ and so that $X \cap R_j \neq \emptyset$ for all $j \leq m$.  
We say that $X$ is \emph{computably compact} if
the set of all minimal covers of $X$ is computably enumerable.

Suppose $f$ is a function that maps complex numbers to complex numbers.  We say that $f$ is \emph{computable} if there is an algorithm $P$ that satisfies the following three criteria.
\begin{itemize}
	\item \bf Approximation:\rm\ Whenever $P$ is given an open rational rectangle as input, it either does not halt or produces an open rational rectangle as output.  (Here, the input rectangle is regarded as an approximation of some $z \in \dom(f)$ and the output rectangle is regarded as an approximation of $f(z)$.)

	\item \bf Correctness:\rm\ Whenever $P$ halts on an open rational rectangle $R$, the rectangle it outputs contains $f(z)$ for each $z \in R \cap \dom(f)$. 
	
	\item \bf Convergence:\rm\ Suppose $U$ is a neighborhood of a point $z \in \dom(f)$ and that $V$ is a neighborhood of $f(z)$.  Then, there is an open rational rectangle $R$ such that $R$ contains $z$, $R$ is included in $U$, and when $R$ is put into $P$, $P$ produces a rational rectangle that is included in $V$.
\end{itemize}
For example, $\sin$, $\cos$, and $\exp$ are computable as can be seen by considering their power series expansions and the bounds on the convergence of these series that can be obtained from Taylor's Theorem.  A consequence of this definition is that computable functions on the complex plane must be continuous.  

A sequence of functions $\{f_n\}_{n \in \N}$ of a complex variable is computable if it is computable uniformly in $n$; that is, there is an algorithm that given any $n \in \N$ as input produces an algorithm that computes $f_n$.

It is well known that integration is a computable functional on $C[0,1]$.  It follows that 
when $f$ is a computable analytic function on the unit disk, the sequence $\{a_n(f)\}_{n \in \N}$ is 
computable uniformly in $f$.  It also follows that $\Log$ is computable.

A \emph{modulus of convergence} for a sequence $\{a_n\}_{n \in \N}$ of points in a complete metric space $(X, d)$ is a function $g : \N \rightarrow \N$ so that $d(a_m, a_n) < 2^{-k}$ whenever $m,n \geq g(k)$.  Thus, a sequence of points in a complete metric space converges if and only if it has a modulus of convergence.  Suppose $p \geq 1$ is computable.  If $\{f_n\}_{n \in \N}$ is a computable and convergent sequence of vectors in $L^p[-\pi, \pi]$, then $\lim_n f_n$ is a computable vector if and only if $\{f_n\}_{n \in \N}$ has a \emph{computable} modulus of convergence.

Suppose $f$ is a uniformly continuous computable function that maps complex numbers to complex numbers.  
A \emph{modulus of uniform continuity} for $f$ is a function $m : \N \rightarrow \N$ so that 
$|f(z_0) - f(z_1)| < 2^{-k}$ whenever $z_0, z_1 \in \dom(f)$ and $|z_0 - z_1| \leq 2^{-m(k)}$.   
If the domain of $f$ is computably compact, then $f$ has a computable modulus of uniform continuity.

Suppose $\{a_n\}_{n \in \N}$ is a sequence of complex numbers so that 
$\sum_{n = 0}^\infty a_n z^n$ converges whenever $|z| < 1$, and suppose 
$G$ is a compact subset of the unit disk.  A \emph{modulus of uniform convergence} for this series on $G$ is a function $m : \N \rightarrow \N$ so that $\left| \sum_{n = m(k)}^\infty a_n z^n\right| < 2^{-k}$ whenever $z \in G$ and $k \in \N$.  If the sequence $\{a_n\}_{n \in \N}$ is computable and if $G$ is computably compact, then the series $\sum_{n = 0}^\infty a_n z^n$ has a computable modulus 
of uniform convergence on $G$.

We note that when $p \geq 1$ is a computable real and $f \in L^p[-\pi, \pi]$, there are
two senses in which $f$ can be ``computable'': as a vector and as a function.  These fail to
coincide.  By definition, a computable function is continuous.  However, there are discontinuous functions in $L^1[-\pi, \pi]$ that are computable as vectors; e.g., the greatest integer function. 
Moreover, there are continuous functions in $L^1[-\pi, \pi]$ that are computable as vectors but not as functions.  

Lastly, a \emph{lower semicomputable function} is a function $T:[-\pi,\pi]\rightarrow [0,\infty]$ that is the sum of a computable sequence of nonnegative real-valued functions.

\subsection{Algorithmic randomness}

There are three different approaches to defining the concept of randomness formally. The one we will find useful for this paper is the measure-theoretic one: A random point in a given probability space is said to be random if it avoids all null classes generated in a certain way by computably enumerable functions. Thus, for any reasonable randomness notion, the class of random points is conull. For a general introduction to algorithmic randomness, see \cite{dhbook} or \cite{niesbook}.

While the most-studied randomness notion is Martin-L\"of randomness, a weaker notion, Schnorr randomness, lies at the heart of our paper. Schnorr randomness, like most other randomness notions, was originally defined in the Cantor space $2^\omega$ with Lebesgue measure \cite{schnorr}; however, the definition is easily adaptable to any computable measure space, in particular $[-\pi,\pi]$ with the Lebesgue measure $\mu$.

\begin{definition}\label{def:Schnorr-test}
	A \emph{Schnorr test} is a computable sequence $\{ V_n\}_{n \in \N}$ of open sets of reals so that $\mu(V_n)\leq 2^{-n}$ for all $n$ and so that the sequence $\{\mu(V_n)\}_{n \in \N}$ is computable. A real number $x$ is said to be \emph{Schnorr random} if for every Schnorr test $\{ V_n\}_{n \in \N}$, $x\not\in \bigcap_n V_n$.
\end{definition}

There are many other characterizations of Schnorr randomness, such as a complexity-based characterization \cite{dg04} and a martingale characterization \cite{schnorr}. In this paper, we will use an integral test characterization due to Miyabe \cite{miyabe13} which is rooted in computable analysis.

\begin{definition}\label{def:int-test}
	A \emph{Schnorr integral test} is a lower semicomputable function $T:[-\pi,\pi]\rightarrow [0,\infty]$ so that $\int_{-\pi}^\pi T\ d\mu$ is a computable real.
\end{definition}

Thus, if $T$ is a Schnorr integral test, then $T(x)$ is finite for almost every $x \in [-\pi, \pi]$.  
Miyabe's characterization states that $x \in [-\pi, \pi]$ is Schnorr random if and only if 
$T(x) < \infty$ for every Schnorr integral test $T$.

\section{Proof of Theorem \ref{thm:main.1}}\label{sec:1stproof}

Our proof of Theorem \ref{thm:main.1} is based on the following definition and lemmas.

\begin{definition}\label{def:modulus.ae.conv}
Suppose $\{f_n\}_{n \in \N}$ is a sequence of functions on $[-\pi, \pi]$.  
A function $\eta : \N \times \N \rightarrow \N$ is a \emph{modulus of almost-everywhere convergence} for $\{f_n\}_{n \in \N}$ if 
\[
\mu(\{t \in [-\pi, \pi]\ :\ \exists M,N \geq \eta(k,m)\ |f_N(t) - f_M(t)| \geq 2^{-k}\}) < 2^{-m}
\]
for all $k$ and $m$.
\end{definition}

Thus, every sequence of functions on $[-\pi, \pi]$ that converges almost everywhere
has a modulus of almost-everywhere convergence.  Our goal, as stated in the following lemma, 
is to show that the sequence of partial sums for the Fourier series of a computable vector in $L^p[-\pi, \pi]$ has a \emph{computable} modulus of almost-everywhere convergence.  

\begin{lemma}\label{lm:modulus.ae.conv}
Suppose $p$ is a computable real so that $p > 1$, and suppose $f$ is a computable vector 
in $L^p[-\pi, \pi]$.  Then, $\{S_N(f)\}_{N \in \N}$ has a computable modulus of almost-everywhere convergence.
\end{lemma}

With this lemma in hand, Theorem \ref{thm:main.1} follows from the next lemma.

\begin{lemma}\label{lm:effective-convergence}
Assume $\{f_n\}_{n \in \N}$ is a uniformly computable sequence of functions on $[-\pi, \pi]$ for which there is a computable modulus of almost-everywhere convergence.
Then, the sequence $\{f_n\}_{n \in \N}$ converges at every Schnorr random real.
\end{lemma}

Generalizations of Lemma \ref{lm:effective-convergence} can be found in Galatolo, Hoyrup, and Rojas \cite[Thm.~1]{Galatolo.Hoyrup.Rojas:2010a}  and well as Rute \cite[Lemma~3.19 on p.~41]{Rute:2013pd}.  Our proof is new.  Theorem \ref{thm:main.1} follows by applying Lemma \ref{lm:effective-convergence} to the sequence of partial sums of $f$.

\begin{proof}[Proof of Lemma \ref{lm:modulus.ae.conv}]
We compute $\eta : \N^2 \rightarrow \N$ as follows.  
Let $k,m \in \N$ be given as input.  Compute a rational trigonometric polynomial $\tau_{k,m}$ 
so that $\norm{f - \tau_{k,m}}_p \leq 2^{-(m+ k + 3)}C^{-1}$ where $C$ is as in Fefferman's inequality.  Then define $\eta(k,m)$ to be the degree of $\tau_{k,m}$.

By definition, $\eta$ is computable.  We now show that it is a modulus of almost-everywhere convergence.  We begin with some notation.  Let $g \in L^p[-\pi, \pi]$.  Set
\[
E_{k, N_0}(g) = \{t \in [-\pi, \pi]\ :\ \exists M, N \geq N_0\ |S_M(g)(t) - S_N(g)(t)| \geq 2^{-k}\}.
\]
Thus, we aim to show that $\mu(E_{k,\eta(k,m)}(f)) < 2^{-m}$.
For each $k \in \N$, let 
\[
\hat{E}_k(g) = \{t \in [-\pi, \pi]\ :\ \sup_N |S_N(g)(t)| > 2^{-k}\}.
\]
It follows that $E_{k, N_0}(g) \subseteq \hat{E}_{k+2}(g)$.

We claim that 
$E_{k, \eta(k,m)}(f) \subseteq E_{k, \eta(k,m)}(f - \tau_{k,m})$.  We see that if $M, N \geq \eta(k,m)$, then 
\begin{eqnarray*}
|S_M(f)(t) - S_N(f)(t)| & \leq & |S_M(f - \tau_{k,m})(t) - S_N(f - \tau_{k,m})(t)| + |S_M(\tau_{k,m})(t) - S_N(\tau_{k,m})(t)|\\
& = & |S_M(f - \tau_{k,m})(t) - S_N(f - \tau_{k,m})(t)|.
\end{eqnarray*}	
Thus, $E_{k,\eta(k,m)}(f) \subseteq E_{k,\eta(k,m)}(f - \tau_{k,m})$.

It now follows that  
$E_{k, \eta(k,m)}(f) \subseteq \hat{E}_{k + 2}(f - \tau_{k,m})$.  
We complete the proof by showing that $\mu(\hat{E}_{k+2}(f - \tau_{k,m})) < 2^{-m}$.
By 
Fefferman's Inequality and the definition of $\tau_{k,m}$, 
\[
\norm{\sup_N|S_N(f - \tau_{k,m})|}_1 \leq 2^{-(m + k + 3)}.
\]
Thus, by Chebyshev's Inequality, 
\[
\mu(\hat{E}_{k+2}(f - \tau_{k,m})) \leq 2^{-(m + k + 3)}2^{k+2} = 2^{-(m+1)} < 2^{-m}.
\]
Hence, $\mu(\hat{E}_{k+2}(f - \tau_{k,m})) < 2^{-m}$.
\end{proof}

\begin{proof}[Proof of Lemma \ref{lm:effective-convergence}]  We apply Miyabe's characterization of 
Schnorr randomness.  We begin by defining a Schnorr integral test as follows.
Let $\eta$ be a computable modulus of almost-everywhere convergence for $\{f_n\}_{n \in \N}$, and  
abbreviate $\eta(k,k)$ by $N_k$. For each $k \in \N$ and each $t \in [-\pi, \pi]$ define 
\[
g_k(t) = \min\left\{1,\ \  \max \{|f_M(t) - f_N(t))| : N_k < M,N \leq N_{k+1}\}\right\}.
\] 
The sequence $\{g_k\}_{k \in \N}$ is computable.  
Set $T = \sum_{k = 0}^\infty g_k$.  

We now show that $T$ is a Schnorr integral test.  By construction, $T$ is lower semicomputable.  
Therefore, it suffices to show that $\int_{-\pi}^\pi T d\mu$ is a computable real.  To this end, let $m \in \N$ 
be given.  Since $g_k$ is computable uniformly in $k$, it is possible to compute a rational number $q$ so that 
$|\sum_{k = 0}^{m+6} g_k - q| < 2^{-(m+1)}$.  We claim that 
$|\int_{-\pi}^\pi T d\mu - q| < 2^{-m}$.  By the Monotone Convergence Theorem, 
\[
\int_{-\pi}^\pi T\ d\mu = \sum_{k = 0}^\infty \int_{-\pi}^\pi g_k\ d\mu.
\]
Since $g_k \leq 1$ and
$\mu\{t : g_k(t) \geq 2^{-k}\} \leq 2^{-k}$, we have
\begin{align*}
\int_{-\pi}^\pi g_k(t)\,dt &\leq  2^{-k} \cdot \mu\{t : g(t) < 2^{-k}\} + 1 \cdot \mu\{t : g(t) \geq 2^{-k}\} \\
&\leq 2^{-k}\cdot 2\pi + 1\cdot 2^{-k}  \leq 2^{-k+4}.
\end{align*}
Thus, $\int_{-\pi}^\pi T\ d\mu - \sum_{k = 0}^{m+6} g_k < 2^{-(m+1)}$, and we then have  
$|\int_{-\pi}^\pi T\ d\mu - q| < 2^{-m}$.  Hence, $\int_{-\pi}^\pi  T\ d\mu$ is a computable real.

Finally, we show that $T(t_0) = \infty$ whenever $\{f_n(t_0)\}_{n \in \N}$ diverges.  This will complete
the proof of the lemma.  Suppose $\{f_n(t_0)\}_{n \in \N}$ diverges.  Then there exists $k_0 \in \N$ 
so that $\limsup_{M,N} |f_M(t_0) - f_N(t_0)| \geq 2^{-k_0}$.  
It thus suffices to show that $\sum_{k = k_1}^\infty g_k(t_0) \geq 2^{-k_0}$ for all $k_1 \in \N$.  
So, let $k_1 \in \N$.  Without loss of generality, suppose $g_k < 1$ whenever $k \geq k_1$.  By the choice of $k_0$, there exist $M$ and $N$ so that $N_{k_1} \leq M < N$ and $|f_M(t_0) - f_N(t_0)| \geq 2^{-k_0}$.  By forming a telescoping sum and applying the triangle inequality we obtain
\[
|f_M(t_0) - f_N(t_0)| \leq \sum_{k = k_1}^\infty g_k(t_0).
\]
Thus $T(t_0) = \infty$, and the proof is complete.
\end{proof}

\section{Proof of Theorem \ref{thm:main.2}}\label{sec:2ndproof}

Our proof of Theorem \ref{thm:main.2} is based on a construction of Kahane and Katznelson \cite{Kahane.Katznelson.1966} and requires the following sequence of three lemmas.

\begin{lemma}\label{lm:analytic}
Suppose $G$ is a computably compact subset of the unit circle so that $\lambda(G)$ is computable 
and smaller than $2 \pi$.  Then there is a computable function $\psi$ from $\D \cup G$ into the horizontal
strip $\R \times (-\frac{\pi}{2}, \frac{\pi}{2})$ that is analytic on $\D$ and has the property 
that $\Re(\psi(\zeta)) \geq -\frac{3}{4} \ln(\lambda(G)(2\pi)^{-1})$ for all $\zeta \in G$.  Furthermore, 
we may choose $\psi$ so that $\psi(0) = 0$.  
\end{lemma}

\begin{lemma}\label{lm:poly.1}
Suppose $G$ is a computably compact subset of $[-\pi, \pi]$ so that $\lambda(G)$ is computable 
and smaller than $2\pi$.  Then there is a computable and analytic trigonometric polynomial $R$
so that $\Re(R(t)) \geq -\frac{1}{2} \ln(\lambda(G)/ (2\pi))$ for all $t \in G$ and so that 
$|\Im(R(t))| < \pi$ for all $t \in [-\pi, \pi]$.  Furthermore, we may choose $R$ so that 
$R(0) = 0$. 
\end{lemma}

\begin{lemma}\label{lm:poly.2}
Suppose $G$ is a computably compact subset of $[-\pi, \pi]$ so that $\lambda(G)$ is computable 
and smaller than $2\pi$.  Then there is a computable trigonometric polynomial $p$ so that 
\[
\sup_N |S_N(p)(t)| \geq -\frac{1}{4\pi} \ln\left(\frac{\lambda(G)}{2\pi}\right)
\]
for all $t \in G$ and so that $\norm{p}_\infty < 1$. 
\end{lemma}

\begin{proof}[Proof of Lemma \ref{lm:analytic}]
By first applying a rotation if necessary, we can assume that $-1 \not \in G$.  
Set $a = \lambda(G) / (2\pi)$.  Since $G$ is computably compact, we can compute 
pairwise disjoint open subarcs of the unit circle $A_0, \ldots, A_s$ so that $G \subseteq \bigcup_{j \leq s} A_j$
and so that 
\[
\frac{\lambda(\bigcup_j A_j)}{2\pi} \leq a^{3/4}.
\]
Set $F = \bigcup_{j \leq s} A_j$ and set $a' = \lambda(F) / (2\pi)$.  For each 
$z \in \D \cup G$, let $\phi(x) = \hat{\omega}(z, F, \D)$.  Thus, $\phi$ is computable and 
$\phi$ is analytic in $\D$.  By Equation (\ref{eqn:omega.hat}), $\phi(0) = a'$.  
It also follows that $\Re(\phi(z)) > 0$ for all $z \in \D \cup G$.  So, for all $z \in \D \cup G$, 
set $\psi_1(z) = \Log(\phi(z))$.  Thus, $\psi_1(0) = \ln(a')$.  Set $\psi(z) = \psi_1(z) - \ln(a')$ for 
all $z \in \D \cup G$.  Hence, $\psi_1$ and $\psi$ are computable.  

Let $\zeta \in G$.  We claim that $\Re(\psi(\zeta)) \geq -\frac{3}{4} \ln(a)$.   
We begin by noting that $\Re(\psi(\zeta)) = \Re(\psi_1(\zeta)) - \ln(a')$.  By our choice of $F$, 
$\ln(a') - \frac{3}{4} \ln(a) \leq 0$.  Thus, $\Re(\psi(\zeta)) \geq -\frac{3}{4} \ln(a)$ for 
all $\zeta \in G$.  

Since $\Re(\phi) > 0$, it follows that $|\Im(\psi_1(z))| < \frac{\pi}{2}$.  However, $\Im(\psi) = \Im(\psi_1)$ since 
$\ln(a')$ is real. 
\end{proof}

We note that the proof of Lemma \ref{lm:analytic} is uniform. 

\begin{proof}[Proof of Lemma \ref{lm:poly.1}]
Let $G' = \{e^{it}\ :\ t \in G\}$.  Thus, $G'$ is a computably compact subset of the unit 
circle and $\lambda(G') < 2\pi$.  Let $\psi$ be as given by Lemma \ref{lm:analytic}.  

Let 
\[
G'' = \{r \zeta\ :\ 0 \leq r \leq 1\ \wedge\ \zeta \in G'\}.
\]
It follows that $G''$ is computably compact.  Thus, $\psi$ is uniformly continuous on $G''$ and 
has a computable modulus of uniform continuity on $G''$.  This means that we can compute a rational number 
$r_0 \in (0,1)$ so that 
\[
\Re(\psi(r_0 \zeta)) \geq -\frac{5}{8} \ln\left(\frac{\lambda(G')}{2\pi}\right)
\] 
for all $\zeta \in G$.  

We now abbreviate $a_n(\psi)$ by $a_n$.  
Let $G^{(3)} = \{r_0 \zeta\ :\ \zeta \in G'\}$.  The series $\sum_{n = 0}^\infty a_n z^n$ converges
uniformly on $G^{(3)}$, and we can compute a modulus of uniform convergence for 
it on $G^{(3)}$.  It follows that we can compute $N$ so that for all $\zeta \in G'$, 
\[
\Re\left(\sum_{n = 0}^N a_n r_0^n \zeta^n\right) \geq -\frac{1}{2} \ln\left(\frac{\lambda(G')}{2\pi}\right)\mbox{ and}
\]
\[
\left|\Im\left(\sum_{n = 0}^N a_n r_0^n \zeta^n\right)\right| < \pi.
\]
Set $R(t) = \sum_{n = 0}^N a_n r_0^n e^{i n t}$.
\end{proof}

We note that the proof of Lemma \ref{lm:poly.1} is also uniform.

\begin{proof}[Proof of Lemma \ref{lm:poly.2}]
Let $R$ be as given in Lemma \ref{lm:poly.1} and let $N = \deg(R)$.  Set $q = \Im(R)$ and $p = \frac{1}{\pi} e_{-N} \cdot q$.  

We claim that $|S_N(p)| = \frac{1}{2\pi} |R|$.  For convenience, we abbreviate $c_n(R)$ by $c_n$.  Then $c_m = 0$ when $m \leq 0$ and
\begin{eqnarray*}
q(t) & = & \frac{1}{2i} \left[ \sum_{n = 1}^N c_n e^{i n t} + \sum_{n = 1}^N \overline{(-c_n)} e^{-int} \right]\mbox{ and}\\
p(t) & = & \frac{1}{2\pi i} \left[ \sum_{m = 1 - N}^0 c_{m + N} e^{i nt} + \sum_{m = 1 + N}^{2N} (\overline{-c_{m - N}}) e^{-int}\right].
\end{eqnarray*}
Therefore, 
\[
S_N(p)(t) = e^{-iNt} \frac{1}{2\pi} S_N(R)(t).
\]
Thus, $|S_N(p)| = \frac{1}{2\pi} |R|$. 

Therefore, 
\[
\sup_N|S_N(p)(t)| \geq - \frac{1}{4\pi} \ln\left(\frac{\mu(G)}{2\pi}\right)
\]
for all $t \in G$.  

Since $|\Im(R)| < \pi$, it follows that $\norm{p}_\infty < 1$.  
\end{proof}

Note that the proof of Lemma \ref{lm:poly.2} is uniform as well. 

Now suppose $t_0$ is not Schnorr random.  Then there is a Schnorr test 
$\{U_n\}_{n \in \N}$ so that $t_0 \in \bigcap_n U_n$.  

We construct an array of trigonometric polynomials $\{p_{n,k}\}_{n,k \in \N}$ as follows.  
Since $U_{2^n}$ is computably open uniformly in $N$, we can compute an array of closed rational
intervals $\{I_{n,j}\}_{n,j \in \N}$ so that $U_{2^n} = \bigcup_j I_{n,j}$ and so that 
$\mu(I_{n,j} \cap I_{n,j'}) = 0$ when $j \neq j'$.
We then compute for each $n \in \N$ an increasing sequence $m_{n,0} < m_{n,1} < \ldots$
so that 
\[
\mu\left(U_{2^n} - \bigcup_{j \leq m_{n,k}} I_{n,j}\right) < 2^{-(2^{n + k + 1})}
\]
for all $n$ and $k$. We define the following sets:
\begin{eqnarray*}
G_{n,0} & = & \bigcup_{j \leq m_{n,0}} I_{n,j} \cap [-\pi, \pi]\mbox{ and}\\
G_{n,k} & = & \bigcup_{m_{n,k} < j \leq m_{n,k+1}} I_{n,j} \cap [-\pi,\pi].
\end{eqnarray*}
It follows that 
\[
\mu(G_{n,k}) < 2^{-(2^{n+k})}
\]
for all $n$ and $k$.

Now fix $n$ and $k$.  By Lemma \ref{lm:poly.2}, we can compute a trigonometric polynomial $p$ so that $\norm{p}_\infty < 1$
and 
\[
\sup_N |S_N(p)(t)| > -\frac{1}{4\pi} \ln\left( \frac{\mu(G_{n,k})}{2\pi}\right)
\]
for all $t \in G_{n,k}$.  Set $p_{n,k} = 2^{-(n+ k+1)}p$.  It follows that $\sup_N| S_N(p_{n,k})(t)| > (8\pi)^{-1}$
for all $t \in G_{n,k}$.  

We can now compute an array of nonnegative integers $\{r_{n,k}\}_{n,k \in \N}$ so that 
for each $m \in \Z$, either $c_m(e_{r_{n,k}}\cdot p_{n,k})=0$ or $c_m(e_{r_{n',k'}}\cdot p_{n',k'})=0$ whenever $(n,k) \neq (n',k')$ and so that $c_m(e_{r_{n,k}}\cdot p_{n,k}) = 0$ whenever $m < \langle n,k \rangle$. We set 
\begin{eqnarray*}
f_{n,k} & = & e_{r_{n,k}} \cdot p_{n,k}\mbox{ and}\\
f & = & \sum_{n,k} f_{n,k}.
\end{eqnarray*}
Since $\norm{p_{n,k}}_\infty < 2^{-(n+k+1)}$, it follows that $f$ is computable.  

We now show that the Fourier series of $f$ diverges at $t_0$.  It suffices to show that 
\[
\sup_{M,N} |S_M(f)(t_0) - S_N(f)(t_0)| > \frac{1}{8 \pi}.
\]
Let $N_0 \in \N$ and choose $n$ so that $\langle n,0\rangle \geq N_0$ and $k$ so that $t_0 \in G_{n,k}$.  
By the construction of $\{r_{n,k}\}_{n,k \in \N}$ there exist $M,N'$ so that 
$f_{n,k} = S_{N'}(f) - S_M(f)$ and $M \geq \langle n,k \rangle \geq \langle n, 0 \rangle$.
By the construction of $p_{n,k}$, there exists $N$ so that $M \leq N \leq N'$ and 
$|S_N(f)(t_0) - S_M(f)(t_0)| > (8\pi)^{-1}$.

Thus, the Fourier series for $f$ diverges at $t_0$.

\section{A strengthening of Theorem \ref{thm:main.1}; convergence \emph{to} $f(t_0)$}\label{sec:conv.to.ft0}

Throughout this subsection, $p$ denotes a computable real such that $p \geq 1$.

In Theorem~\ref{thm:main.1} we showed that $S_N(f)(t_0)$ converges for Schnorr randoms $t_0$ and computable vectors $f \in L^p[-\pi,\pi]$.  However, one would like to also say that $\{S_N(f)(t_0)\}_{N \in \N}$ converges to $f(t_0)$.  The problem is that $f$ is merely a vector in $L^p[-\pi,\pi]$, and so\ $f(t_0)$ is not well defined.  Recall that a vector in $L^p[-\pi, \pi]$ is actually an equivalence
class of functions under the ``equal almost everywhere'' relation, and so every complex number is a candidate for the value of $f(t_0)$.  Thus, the limit of $\{S_N(f)(t_0)\}_{N \in \N}$ may not be $f(t_0)$ 
even if $t_0$ is Schnorr random.%
\footnote{For example, Pour-El and Richards \cite[p.~114]{Pour-El.Richards:1989} remark, ``Of course, pointwise evaluation makes no sense for $L^{p}$-functions, since an $L^{p}$-function is only determined up to sets of measure zero. This limitation already exists in classical analysis, without any notions of logical `effectiveness' being required. By its very nature, an $L^{p}$-function is known only on the average.''} %
However, this problem has a solution via Cauchy names, a standard device in computable analysis.  

\begin{definition}\label{def:name}
A sequence of rational trigonometric polynomials $\{\tau_n\}_{n \in \N}$ is a 
\emph{Cauchy name} of a vector $f \in L^p[-\pi, \pi]$ if 
$\lim_{n \rightarrow \infty} \norm{\tau_n - f}_p = 0$ and if 
$\norm{\tau_n - \tau_{n+1}}_p < 2^{-(n+1)}$ for all $n \in \N$.
\end{definition}

Thus, a Cauchy name of a vector in $L^p[-\pi, \pi]$ is a name of exactly one such vector
(up to almost everywhere equality).   

A Cauchy name $\{\tau_n\}_{n \in \N}$ of a vector is \emph{computable} if $\{\tau_n\}_{n \in \N}$ 
is a uniformly computable sequence of rational trigonometric polynomials in the sense 
that the degree and coefficients of $\tau_n$ can be computed from $n$.  It follows that a vector 
in $L^p[-\pi, \pi]$ is computable if and only if it has a computable Cauchy name.

We will show that there is a natural way to use a computable Cauchy name of $f \in L^p[-\pi, \pi]$
to assign a canonical value to $f(t_0)$ when $t_0$ is Schnorr random.  We will then show that 
if $f$ is a computable vector in $L^p[-\pi, \pi]$, then $\lim_{N \rightarrow \infty} S_N(f)(t_0)$ \emph{is} the canonical value of $f(t_0)$ whenever $t_0$ is Schnorr random.  Our approach is based on the following theorem which effectivizes a well-known result in measure theory; namely, that a convergent sequence in $L^p[-\pi, \pi]$ has a subsequence that converges almost everywhere.

\begin{theorem}\label{thm:eff.RF}
Suppose $f_n \in L^p[-\pi, \pi]$ for all $n \in \N$ and suppose $g$ is a computable modulus of 
convergence for $\{f_n\}_{n \in \N}$.  Then, $\eta(k,m) = \lceil \frac{1}{2}(\frac{m+1}{p} + k + 1)\rceil$
defines a modulus of almost-everywhere convergence for $\{f_{g(2n)}\}_{n \in \N}$.
\end{theorem}

\begin{proof}
Set 
\[
E_{n,r} = \{t \in [-\pi, \pi]\ :\ |f_{g(2n + 1)}(t) - f_{g(2n)}(t)| \geq 2^{-r}\}.
\]
Since $g$ is a modulus of convergence for $\{f_n\}_{n \in \N}$, $\norm{f_{g(2n+1)} - f_{g(2n)}}_p^p < 2^{-2np}$.  Thus, by Chebychev's Inequality, $\mu(E_{n,r}) \leq 2^{p(r - 2n)}$.  

Set $N_0 = \eta(k,m)$.  Suppose $M, N \geq N_0$ and 
$|f_{g(2M)}(t) - f_{g(2N)}(t)| \geq 2^{-k}$.  Then
\[
2^{-k} \leq \sum_{n = N_0}^\infty |f_{g(2m)}(t) - f_{g(2n)}(t)|.
\]
It follows that $t \in \bigcup_{c = 0}^\infty E_{N_0 + c, k+1+c}$.  
But, by the definition of $\eta$, 
\[
 \sum_{c = 0}^\infty 2^{p(k + 1 + c - 2(N_0 + c)}\\
 <  \sum_{c = m+1}^\infty 2^{-c} = 2^{-m}.
\]
Thus, $\mu(\bigcup_{c = 0}^\infty E_{N_0 + c, k+1+c}) < 2^{-m}$.  
It follows that $\eta$ is a modulus of almost-everywhere convergence for $\{f_{g(2n)}\}_{n \in \N}$.
\end{proof}

\begin{corollary}\label{cor:name.convergence}
If $\{\tau_n\}_{n \in \N}$ is a computable Cauchy name for a vector in $L^p[-\pi, \pi]$ and if 
$t_0 \in [-\pi, \pi]$ is Schnorr random, then $\{\tau_{2n}(t_0)\}_{n \in \N}$ converges.
\end{corollary}

Corollary \ref{cor:name.convergence} leads to the idea that a Cauchy name for $f$ assigns a value  
to $f(t)$ for Schnorr random $t$.  

\begin{definition}\label{def:value.assigned}
If $\{\tau_n\}_{n \in \N}$ is a computable Cauchy name for a vector $f \in L^p[-\pi, \pi]$, 
and if $\{\tau_{2n}(t)\}_{n \in \N}$ converges to $\alpha \in \C$, then we say 
$\{\tau_n\}_{n \in \N}$ \emph{assigns the value $\alpha$ to $f(t)$}.
\end{definition}

Thus, if $f$ is a computable vector in $L^p[-\pi, \pi]$ and if $t_0$ is Schnorr random, 
then a value is assigned to $f(t_0)$ by each computable Cauchy name of $f$.  We now show that 
the \emph{same} value is assigned by \emph{all} computable Cauchy names via the following proposition.

\begin{proposition}\label{prop:same.limit}
Suppose $\{f_n\}_{n \in \N}$ and $\{g_n\}_{n \in \N}$ are computable sequence of functions on 
$[-\pi, \pi]$ and that each has a computable modulus of almost-everywhere convergence. 
Suppose also that $\lim_{n \rightarrow \infty} f_n(t) - g_n(t) = 0$ for almost every 
$t \in [-\pi, \pi]$.  Then, $\lim_{n \rightarrow \infty} f_n(t_0) = \lim_{n \rightarrow \infty} g_n(t_0)$ whenever $t_0 \in [-\pi, \pi]$ is Schnorr random. 
\end{proposition} 

\begin{proof}
Let $\eta_0$ be a computable modulus of almost-everywhere convergence for $\{f_n\}_{n \in \N}$, and 
let $\eta_1$ be a computable modulus of almost-everywhere convergence for $\{g_n\}_{n \in \N}$.  
Let $h_{2n} = f_n$, and let $h_{2n + 1} = g_n$.  Set 
$\eta(k,m) = \eta_0(k+1, m+2) + \eta_1(k+1, m+2)$.  It follows that 
$\eta$ is a computable modulus of almost-everywhere convergence for $\{h_n\}_{n \in \N}$.  
So, if $t_0$ is Schnorr random, then $\{h_n(t_0)\}_{n \in \N}$ converges, and so 
$\{f_n(t_0)\}_{n \in \N}$ and $\{g_n(t_0)\}_{n \in \N}$ converge to the same value.  
\end{proof}

\begin{definition}\label{def:canonical.value}
Suppose $f$ is a computable vector in $L^p[-\pi, \pi]$.  When $t_0$ is Schnorr random, the
\emph{canonical value} of $f(t_0)$ is the value assigned to $f(t_0)$ by a computable Cauchy
name of $f$.
\end{definition}

By Proposition \ref{prop:same.limit}, the choice of computable Cauchy name does not matter.
Note that if $f$ is continuous, then the canonical value of $f(t_0)$ is in fact $f(t_0)$.

Proposition \ref{prop:same.limit} yields an extension of Theorem \ref{thm:main.1}.

\begin{corollary}\label{cor:conv.to.value}
Suppose $p > 1$ and suppose $f$ is a computable vector in $L^p[-\pi, \pi]$.  
Then, $\{S_N(f)(t_0)\}_{N \in \N}$ converges to the canonical value of $f(t_0)$ 
whenever $t_0$ is Schnorr random.
\end{corollary}

It should also be remarked that these canonical values are similar to Miyabe's Schnorr layerwise computable functions from \cite{miyabe13}.

\section{The $p = 1$ case; characterizing Schnorr randomness via the Fej\'er-Lebesgue Theorem}\label{sec:Fejer.Lebesgue}


Carleson's Theorem does not hold for vectors in $L^1[\pi,\pi]$.  Indeed, Kolmogorov \cite{Kolmogorov:1923} constructed a complex-valued function $f$ in $L^1[\pi,\pi]$ for which $\{S_N(f)(t)\}_{N \in \N}$ diverges almost everywhere (later improved to ``diverges everywhere'').  Moser \cite{Moser:2010} further constructed a computable such $f$. Nonetheless, Fej\'er and Lebesgue  proved that the Ces\'aro means of $\{S_N(f)\}_{N \in \N}$ converge to $f$ almost everywhere. In this section, we will show that the exceptional set of this theorem also characterizes Schnorr randomness.  
We begin by reviewing the relevant components of the classical theory.  We will then discuss 
their effective renditions.   

Recall that the $(N+1)^{st}$ Ces\'aro mean of a sequence $\{a_n\}_{n \in \N}$ is 
\[
\frac{1}{N+1} \sum_{n = 0}^N a_n.
\]
If $\{a_n\}_{n \in \N}$ converges, then so does the sequence of its Ces\'aro means and to the same limit.  Ces\'aro means provide a widely-used method for ``evaluating divergent series;''  e.g., the Ces\'aro means of the partial sums of $\sum_{n = 0}^\infty (-1)^n$ converge to $\frac{1}{2}$.  

Now fix a vector $f \in L^1[-\pi, \pi]$.  Let $\sigma_N(f)$ denote the $(N+1)^{st}$ Ces\'aro 
mean of $\{S_N(f)\}_{N \in \N}$.  That is, 
\[
\sigma_N(f) = \frac{1}{N + 1} \sum_{M = 0}^N S_M(f).
\]  
One can also express $\sigma_N(f)$ via the convolution
\begin{equation}\label{eqn:Fejer}
\sigma_N(f)(t) = \frac{1}{2\pi}\int_{-\pi}^{\pi}f(t-x)F_N(x)\,dx
\end{equation}
where $F_N$ is the \emph{Fej\'er kernel}
\[
F_N(x)= \frac{1}{N} \frac{\sin^2(Nx/2)}{\sin^2(x/2)}.
\]

Recall that $t_0 \in [-\pi, \pi]$ is a Lebesgue point of $f$ if 
\[
\lim_{\epsilon \rightarrow 0^+} \frac{1}{2\epsilon} \int_{t_0 - \epsilon}^{t_0 + \epsilon} |f(t)-f(t_0)|\ dt = 0.
\]
One of Lebesgue's differentiation theorems states that almost every point in $[-\pi, \pi]$ is a Lebesgue point of $f$.   Building on Fej\'er's work on Ces\'aro means of Fourier series, Lebesgue then showed that $\{\sigma_N(f)(t_0)\}_{N \in \N}$ converges to $f(t_0)$ whenever $t_0$ is a Lebesgue point of $f$ \cite{Lebesgue:1905}.  Fej\'er also showed that $\{\sigma_N(f)\}_{N \in \N}$ converges uniformly if $f$ is 
continuous and periodic (in the sense that $f(\pi) = f(-\pi))$.

The result of this section can now be stated as follows. 

\begin{theorem}\label{thm:eff.Fejer.Lebesgue}
Suppose $t_0 \in [-\pi, \pi]$.  Then, $t_0$ is Schnorr random if and only if 
$\{\sigma_N(f)(t_0)\}_{N \in \N}$ converges to the canonical value of $f(t_0)$ whenever $f$ is a computable vector in $L^1[-\pi, \pi]$.
\end{theorem}

\begin{proof}
Suppose $t_0$ is Schnorr random, and let $f$ be a computable vector in $L^1[-\pi, \pi]$.   
Let $\widetilde{f}$ denote the function so that $\widetilde{f}(t)$ equals the canonical value of $f(t)$
when $t$ is Schnorr random and is $0$ otherwise.  Call $\widetilde{f}$ the \emph{canonical version} of $f$.  Independently, Pathak, Rojas, and Simpson \cite{Pathak.Rojas.Simpson:2014} and Rute \cite{Rute:2013pd} showed that every Schnorr random $t_0$ is a Lebesgue point of $\widetilde{f}$.  Since $\widetilde{f}(t) = f(t)$ almost everywhere, $\sigma_N(\widetilde{f}) = \sigma_N(f)$.  Thus, by the Fej\'er-Lebesgue Theorem, $\{\sigma_N(f)(t_0)\}_{N \in \N}$ converges to the canonical value of $f(t_0)$.  

Now, suppose $t_0$ is not Schnorr random.  Then there is a Schnorr integral test $T$ so that 
$T(t_0) = \infty$.   
We claim that $T$ is a computable vector in $L^1[-\pi, \pi]$.  Suppose 
$T = \sum_{n = 0}^\infty g_n$ where $\{g_n\}_{n \in \N}$ is a computable sequence of 
nonnegative functions.  By the Monotone Convergence Theorem, 
$\norm{T}_1 = \sum_{n = 0}^\infty \norm{g_n}_1$.  Let $k \in \N$ be given.  Since $\norm{T}_1$ is computable, from $k$ we can compute a nonnegative integer $m$ so that 
$\norm{\sum_{n = m+1}^\infty g_n}_1 < 2^{-(k+1)}$.  Since $g_n$ is computable uniformly 
in $n$, we can then compute a trigonometric polynomial $\tau$ so that 
$\norm{\tau - \sum_{n = 0}^m g_n}_1 < 2^{-(k+1)}$.   It follows that 
$\norm{T - \tau}_1 < 2^{-k}$.

We now show that $\lim_N \sigma_N(T)(t_0) = \infty$.  Set $h_k = \sum_{n = 0}^k g_k$.  
Fix $k \in \N$.  Since $F_N \geq 0$, it follows from 
Equation \ref{eqn:Fejer} that $\sigma_N(T)(t_0) \geq \sigma_N(h_k)(t_0)$.  
Since $h_k$ is continuous at $t_0$, $t_0$ is a Lebesgue point for $h_k$ and so 
$\lim_{N \rightarrow \infty} \sigma_N(h_k)(t_0) = h_k(t_0)$.  
Thus, $\lim_N \sigma_N(T)(t_0) \geq h_k(t_0)$.  It follows that 
$\lim_{N \rightarrow \infty} \sigma_N(f)(t_0) = \infty$.
\end{proof}

The forward direction of Theorem \ref{thm:eff.Fejer.Lebesgue} first appeared in Rute's dissertation \cite[Cor.~4.22 on p.~49]{Rute:2013pd}.  Note that if $f : [-\pi, \pi] \rightarrow \C$ is continuous, then every number in $[-\pi, \pi]$ is a Lebesgue point of $f$.  Thus, the converse of Theorem \ref{thm:eff.Fejer.Lebesgue} cannot be made as strong as Theorem \ref{thm:main.2}.  

The proof of the converse of Theorem \ref{thm:eff.Fejer.Lebesgue} can easily be adapted 
to the case where $f \in L^p[-\pi, \pi]$ and $p \geq 1$.  
In addition, the proof of this direction shows that if $T \geq 0$ is a lower semicontinuous and integrable function (possibly with infinite values) and if $p \geq 1$, then there is a vector $f \in L^p[-\pi, \pi]$ so that 
$\{\sigma_N(f)(t)\}_{N \in \N}$ diverges whenever $T(t) = \infty$.  
If $E$ is a measure zero subset of $[-\pi, \pi]$, then there is a lower semicontinuous and non-negative function $T$ so that $\norm{T}_1 < \infty$ and so that 
$T(t) = \infty$ whenever $t \in E$.  We thus obtain the following extension of a result of 
Katznelson by a simpler proof \cite{Katznelson.1966}.

\begin{theorem}
Suppose $p \geq 1$, and suppose $E$ is a measure $0$ subset of $[-\pi, \pi]$.  
Then there exists $f \in L^p[-\pi, \pi]$ so that 
$\{\sigma_N(f)(t)\}_{N \in \N}$ diverges whenever $t \in E$.
\end{theorem}

\section{Conclusion}\label{sec:conclusion}

We have used algorithmic randomness to study an almost-everywhere convergence theorem in analysis.   Many of these theorems have already been investigated, including the ergodic theorem \cite{ft-mp,ghr11,Hoyrup.Rojas:2009b,v97}, the martingale convergence theorem \cite{Rute:2013pd}, the Lebesgue Differentiation Theorem \cite{Pathak.Rojas.Simpson:2014,Rute:2013pd}, Rademacher's Theorem \cite{Freer.Kjos-Hanssen.Nies.ea:2014}, and Lebesgue's theorem concerning the differentiability of bounded variation functions \cite{bmn16}.  This list is not exhaustive and more work needs to be done.  In some cases, the resulting randomness notion is Schnorr randomness.  In others, it is Martin-L\"of randomness or computable randomness.  

In this conclusion, we would like to share some intuition about why Carleson's Theorem characterizes Schnorr randomness and what clues one might look for when investigating similar theorems.  Namely, we are interested in almost-everywhere convergence theorems stating that for a family $\mathcal{F}$ of sequences of functions, every sequence $\{f_n\}_{n \in \N}$ in the family converges almost everywhere.  In Carleson's Theorem, $\mathcal{F}$ is the family of sequences $\{S_N(f)\}_{N \in \N}$ for $f \in L^p$.

The main clue that $\{S_N(f)\}_{N \in \N}$ converges on Schnorr randoms, is that the pointwise limit of this sequence is computable from the parameter $f$ (indeed the limit is $f$).  In such cases where the limit is computable, one can usually (at least from our experience) find a computable modulus of almost-everywhere convergence.  This allows one to apply Lemma~\ref{lm:effective-convergence} or one of its generalizations to show that the sequence $\{f_n\}_{n \in \N}$ converges for Schnorr randoms (e.g. Theorem~\ref{thm:main.1}).  In some cases, this rate of convergence follows from well-known quantitative estimates---Fefferman's Inequality in our case.  Moreover, in convergence theorems where the limit is computable, these theorems are usually constructively provable.  We conjecture that Carleson's Theorem is provable in the logical frameworks of Bishop style constructivism and $\mathsf{RCA}_0$.

On the other hand, if we are working with a theorem, such as the ergodic theorem, where the limit of the theorem is not always computable, then it is unlikely that the sequence $\{f_n\}_{n \in \N}$ converges for all Schnorr randoms. Instead, one should look into weaker randomness notions, such as Martin-L\"of and computable randomness.  Nonetheless, convergence on Schnorr randoms can often be recovered by restricting the theorem. For example, with the ergodic theorem, convergence happens on Schnorr randoms if the system is ergodic (or in any case where the limit is computable).

Lastly, ``reversals'' similar to Theorem~\ref{thm:main.2} are usually effective proofs of a stronger result.  For example, Miyabe's characterization of the Schnorr randoms yields proof of the following principle: If $E \subseteq [-\pi, \pi]$ is a null set, then there is a lower semicontinuous and integrable function $T : [-\pi, \pi] \rightarrow [0, \infty]$ so that $T(t_0) = \infty$ whenever $t_0 \in E$.  If we relativize Theorem~\ref{thm:main.2}, then we get Kahane and Katznelson's result \cite{Kahane.Katznelson.1966} that for every null set $E$, there is a continuous function $f$ such that $\{S_N(f)\}_{N \in \N}$ diverges on $E$.  However, the relativizations of the lemmas in Section \ref{sec:2ndproof} strengthen the intermediate results in \cite{Kahane.Katznelson.1966}, and we have endeavored to carefully justify many important details.  Similarly, if an almost-everywhere convergence theorem characterizes a standard randomness notion, then it usually satisfies the following property: For every null set $E$ there is a sequence $\{f_n\}_{n \in \N}$ for which the theorem says $\{f_n\}_{n \in \N}$ converges almost everywhere, but  $\{f_n\}_{n \in \N}$ diverges on $E$.  Not all almost-everywhere theorems satisfy this property.  Nonetheless, this property does seem to be satisfied by theorems where the parameters of the theorem are functions in $L^p$, such as Carleson's Theorem and the Lebesgue differentiation theorem.

\bibliographystyle{amsplain}

\begin{thebibliography}{10}

\bibitem{abs14}
Kelty Allen, Laurent Bienvenu, and Theodore~A. Slaman, \emph{On zeros of
  {M}artin-{L}\"of random {B}rownian motion}, J. Log. Anal. \textbf{6} (2014),
  Paper 9, 34. 

\bibitem{Asarin.Pokrovskii:1986}
E.~A. Asarin and A.~V. Pokrovski{\u\i}.
\newblock Application of {K}olmogorov complexity to the analysis of the
dynamics of controllable systems.
\newblock {\em Avtomat. i Telemekh.}, (1):25--33, 1986.

\bibitem{av13}
Jeremy Avigad, \emph{Uniform distribution and algorithmic randomness}, J.
  Symbolic Logic \textbf{78} (2013), no.~1, 334--344.

\bibitem{bdhms12}
Laurent Bienvenu, Adam~R. Day, Mathieu Hoyrup, Ilya Mezhirov, and Alexander
  Shen, \emph{A constructive version of {B}irkhoff's ergodic theorem for
  {M}artin-{L}\"of random points}, Inform. and Comput. \textbf{210} (2012),
  21--30. 

\bibitem{bhmn-14}
Laurent Bienvenu, Rupert H{\"o}lzl, Joseph~S. Miller, and Andr{\'e} Nies,
  \emph{Denjoy, {D}emuth and density}, J. Math. Log. \textbf{14} (2014), no.~1,
  1450004, 35.

\bibitem{bmn16}
Vasco Brattka, Joseph~S. Miller, and Andr{\'e} Nies, \emph{Randomness and
  differentiability}, Trans. Amer. Math. Soc. \textbf{368} (2016), no.~1,
  581--605. 

\bibitem{Brattka.Weihrauch.1999}
Vasco Brattka and Klaus Weihrauch, \emph{Computability on subsets of
  {E}uclidean space. {I}. {C}losed and compact subsets}, vol. 219, 1999,
  Computability and complexity in analysis (Castle Dagstuhl, 1997), pp.~65--93.

\bibitem{Braverman.Cook.2006}
M.~Braverman and S.~Cook, \emph{Computing over the reals: foundations for
  scientific computing}, Notices of the American Mathematical Society
  \textbf{53} (2006), no.~3, 318--329.

\bibitem{cf-ud}
Wesley Calvert and Johanna N.~Y. Franklin, \emph{Genericity and {UD}-random
  reals}, J. Log. Anal. \textbf{7} (2015), Paper 4, 10.

\bibitem{carleson66}
Lennart Carleson, \emph{On convergence and growth of partial sums of {F}ourier
  series}, Acta Math. \textbf{116} (1966), 135--157.

\bibitem{Conway.1978}
J.B. Conway.
\newblock {\em Functions of One Complex Variable {I}}, volume~11 of {\em
  Graduate Texts in Mathematics}.
\newblock Springer-Verlag, 2nd edition, 1978.

\bibitem{Cooper.2004}
S.~Barry Cooper, \emph{Computability theory}, Chapman \& Hall/CRC, Boca Raton,
  FL, 2004.

\bibitem{dg04}
Rodney~G. Downey and Evan~J. Griffiths, \emph{Schnorr randomness}, J. Symbolic
  Logic \textbf{69}.

\bibitem{dhbook}
Rodney~G. Downey and Denis~R. Hirschfeldt, \emph{Algorithmic {R}andomness and
  {C}omplexity}, Springer, 2010.

\bibitem{Fefferman.1973}
Charles Fefferman, \emph{Pointwise convergence of {F}ourier series}, Ann. of
  Math. (2) \textbf{98} (1973), 551--571.

\bibitem{Fefferman.1997}
Charles~L. Fefferman, \emph{Erratum: ``{P}ointwise convergence of {F}ourier
  series'' [{A}nn. of {M}ath. (2) {\bf 98} (1973), no. 3, 551--571; {MR}0340926
  (49 \#5676)]}, Ann. of Math. (2) \textbf{146} (1997), no.~1, 239.

\bibitem{Fejer:1900}
L\'eopold Fej\'er.
\newblock Sur les fonctions born\'ees et int\'egrables.
\newblock {\em C. R. Acad. Sci. Paris}, 131:984--987, 1900.

\bibitem{Fouche:2000}
Willem~L. Fouch{\'e}.
\newblock The descriptive complexity of {B}rownian motion.
\newblock {\em Adv. Math.}, 155(2):317--343, 2000.

\bibitem{fgmn}
Johanna~N.Y. Franklin, Noam Greenberg, Joseph~S. Miller, and Keng~Meng Ng,
  \emph{Martin-{L}\"of random points satisfy {B}irkhoff's ergodic theorem for
  effectively closed sets}, Proc. Amer. Math. Soc. \textbf{140} (2012), no.~10,
  3623--3628.

\bibitem{ft-mp}
Johanna~N.Y. Franklin and Henry Towsner, \emph{Randomness and non-ergodic
  systems}, Mosc. Math. J. \textbf{14} (2014), no.~4, 711--744.

\bibitem{Freer.Kjos-Hanssen.Nies.ea:2014}
Cameron Freer, Bj{\o}rn Kjos-Hanssen, Andr{{\'e}} Nies, and Frank Stephan.
\newblock Algorithmic aspects of {L}ipschitz functions.
\newblock {\em Computability}, 3(1):45--61, 2014.

\bibitem{ghr11}
Peter G{\'a}cs, Mathieu Hoyrup, and Crist{\'o}bal Rojas, \emph{Randomness on
  computable probability spaces---a dynamical point of view}, Theory Comput.
  Syst. \textbf{48} (2011), no.~3, 465--485. 

\bibitem{Galatolo.Hoyrup.Rojas:2010a}
Stefano Galatolo, Mathieu Hoyrup, and Crist{\'o}bal Rojas.
\newblock Computing the speed of convergence of ergodic averages and
pseudorandom points in computable dynamical systems.
\newblock In Xizhong Zheng and Ning Zhong, editors, {\em {\rm Proceedings
		Seventh International Conference on} Computability and Complexity in
	Analysis, {\rm Zhenjiang, China, 21-25th June 2010}}, volume~24 of {\em
	Electronic Proceedings in Theoretical Computer Science}, pages 7--18. Open
Publishing Association, 2010.

\bibitem{Garnett.Marshall.2005}
J.~B. Garnett and D.~E. Marshall, \emph{Harmonic measure}, New Mathematical
  Monographs, vol.~2, Cambridge University Press, Cambridge, 2005.

\bibitem{Grzegorczyk.1957}
A.~Grzegorczyk, \emph{On the definitions of computable real continuous
  functions}, Fund. Math. \textbf{44} (1957), 61--71.

\bibitem{hoyrup13}
Mathieu Hoyrup, \emph{Computability of the ergodic decomposition}, Ann. Pure
  Appl. Logic \textbf{164} (2013), no.~5, 542--549. 

\bibitem{Hoyrup.Rojas:2009b}
Mathieu Hoyrup and Crist{{\'o}}bal Rojas.
\newblock Applications of effective probability theory to {M}artin-{L}{\"o}f
randomness.
\newblock In {\em Automata, languages and programming. {P}art {I}}, volume 5555
of {\em Lecture Notes in Comput. Sci.}, pages 549--561. Springer, Berlin,
2009.

\bibitem{hunt68}
Richard~A. Hunt, \emph{On the convergence of {F}ourier series}, Orthogonal
  {E}xpansions and their {C}ontinuous {A}nalogues ({P}roc. {C}onf.,
  {E}dwardsville, {I}ll., 1967), Southern Illinois Univ. Press, Carbondale,
  Ill., 1968, pp.~235--255.

\bibitem{Kahane.Katznelson.1966}
Jean-Pierre Kahane and Yitzhak Katznelson.
\newblock Sur les ensembles de divergence des s\'eries trigonom\'etriques.
\newblock {\em Studia Math.}, 26:305--306, 1966.

\bibitem{Katznelson.1966}
Yitzhak Katznelson.
\newblock Sur les ensembles de divergence des s\'eries trigonom\'etriques.
\newblock {\em Studia Math.}, 26:301--304, 1966.

\bibitem{Kolmogorov:1923}
A.~N. Kolmogorov.
\newblock Une s\'erie de {F}ourier-{L}ebegue divergente presque partout.
\newblock {\em Fund. Math. Fund. Math. Fund. Math.}, 4:324--328, 1923.

\bibitem{Lacombe.1955.a}
Daniel Lacombe, \emph{Extension de la notion de fonction r\'ecursive aux
  fonctions d'une ou plusieurs variables r\'eelles. {I}}, C. R. Acad. Sci.
  Paris \textbf{240} (1955), 2478--2480. 

\bibitem{Lacombe.1955.b}
\bysame, \emph{Extension de la notion de fonction r\'ecursive aux fonctions
  d'une ou plusieurs variables r\'eelles. {II}, {III}}, C. R. Acad. Sci. Paris
  \textbf{241} (1955), 13--14, 151--153. 

\bibitem{Lebesgue:1905}
Henri Lebesgue.
\newblock Recherches sur la convergence des s{\'e}ries de fourier.
\newblock {\em Math. Ann.}, 61(2):251--280, 1905.

\bibitem{miyabe13}
Kenshi Miyabe, \emph{{$L\sp 1$}-computability, layerwise computability and
  {S}olovay reducibility}, Computability \textbf{2} (2013), no.~1, 15--29.
  
\bibitem{mnz}
Kenshi Miyabe, Andr\'{e} Nies, and Jing Zhang, \emph{``{U}niversal'' {S}chnorr
  tests}, In progress.

\bibitem{Moser:2010}
Philippe Moser.
\newblock On the convergence of {F}ourier series of computable {L}ebesgue
integrable functions.
\newblock {\em MLQ Math. Log. Q.}, 56(5):461--469, 2010.

\bibitem{Nehari.1952}
Zeev Nehari.
\newblock {\em Conformal mapping}.
\newblock McGraw-Hill Book Co., Inc., New York, Toronto, London, 1952.

\bibitem{niesbook}
Andr{\'e} Nies, \emph{Computability and {R}andomness}, Clarendon Press, Oxford,
  2009.

\bibitem{nies14}
\bysame, \emph{Differentiability of polynomial time computable
  functions}, 31st {I}nternational {S}ymposium on {T}heoretical {A}spects of
  {C}omputer {S}cience, LIPIcs. Leibniz Int. Proc. Inform., vol.~25, Schloss
  Dagstuhl. Leibniz-Zent. Inform., Wadern, 2014, pp.~602--613. 

\bibitem{o1}
Piergiorgio Odifreddi, \emph{Classical recursion theory}, Studies in Logic and
  the Foundations of Mathematics, no. 125, North-Holland, 1989.

\bibitem{o2}
\bysame, \emph{Classical recursion theory, volume ii}, Studies in Logic and the
  Foundations of Mathematics, no. 143, North-Holland, 1999.

\bibitem{Pathak.Rojas.Simpson:2014}
Noopur Pathak, Crist{{\'o}}bal Rojas, and Stephen~G. Simpson.
\newblock Schnorr randomness and the {L}ebesgue differentiation theorem.
\newblock {\em Proc. Amer. Math. Soc.}, 142(1):335--349, 2014.

\bibitem{Pour-El.Richards:1989}
Marian~B. Pour-El and J.~Ian Richards.
\newblock {\em Computability in analysis and physics}.
\newblock Perspectives in Mathematical Logic. Springer-Verlag, Berlin, 1989.

\bibitem{Rute:2013pd}
Jason Rute.
\newblock {\em Topics in algorithmic randomness and computable analysis}.
\newblock PhD thesis, Carnegie Mellon University, August 2013.
\newblock Available at {\tt http://repository.cmu.edu/dissertations/260/}.

\bibitem{soare}
Robert~I. Soare, \emph{Recursively enumerable sets and degrees}, Perspectives
  in Mathematical Logic, Springer-Verlag, 1987.

\bibitem{schnorr}
C.-P. Schnorr, \emph{Zuf\"{a}lligkeit und {W}ahrscheinlichkeit}, Lecture
  {N}otes in {M}athematics, vol. 218, Springer-Verlag, Heidelberg, 1971.

\bibitem{Specker.1949}
E.~Specker.
\newblock Nicht konstruktiv beweisbare {S}{\"a}tze der {A}nalysis.
\newblock {\em Journal of Symbolic Logic}, 14:145 -- 158, 1949.

\bibitem{Turing.1937}
A.~M. Turing, \emph{On {C}omputable {N}umbers, with an {A}pplication to the
  {E}ntscheidungsproblem. {A} {C}orrection}, Proc. London Math. Soc.
  \textbf{S2-43} (1937), no.~6, 544.

\bibitem{v97}
V.~V. V{\cprime}yugin, \emph{Effective convergence in probability, and an
  ergodic theorem for individual random sequences}, Teor. Veroyatnost. i
  Primenen. \textbf{42} (1997), no.~1, 35--50. 

\bibitem{Weihrauch.2000}
Klaus Weihrauch, \emph{Computable analysis}, Texts in Theoretical Computer
  Science. An EATCS Series, Springer-Verlag, Berlin, 2000.
\end{thebibliography}
\def\cprime{$'$}
\providecommand{\bysame}{\leavevmode\hbox to3em{\hrulefill}\thinspace}
\providecommand{\MR}{\relax\ifhmode\unskip\space\fi MR }
\providecommand{\MRhref}[2]{%
  \href{http://www.ams.org/mathscinet-getitem?mr=#1}{#2}
}
\providecommand{\href}[2]{#2}

\end{document}